\documentclass[10pt]{amsart}
\usepackage{amsmath}
\usepackage{amsthm}
\usepackage{amssymb}
\usepackage{a4wide}
\usepackage{color}
\usepackage[pagewise,mathlines]{lineno}
\usepackage{hyperref}
\usepackage{amsfonts}

\usepackage{mathtools}
\mathtoolsset{showonlyrefs}

\newcommand*{\Mscr}{\mathcal M}
\newcommand*{\Nscr}{\mathcal N}

\newcommand*{\R}{\mathbb R}

\newtheorem{cor}{Corollary}[section]
\newtheorem{prop}{Proposition}[section]
\newtheorem{theorem}{Theorem}[section]

\newtheorem{remark}{Remark}[section]
\newtheorem{example}{Example}[section]

\begin{document}

\title[Hadamard formula and Rayleigh-Faber-Krahn inequality for nonlocal problems]
{The Hadamard formula and the Rayleigh-Faber-Krahn inequality for nonlocal eigenvalue problems}

\author[R. D. Benguria, M. C. Pereira and M. S\'aez]{Rafael D. Benguria, Marcone C. Pereira and Mariel S\'aez}

\address{Rafael D. Benguria
\hfill\break\indent Instituto de F\'isica, P. Universidad Cat\'olica de Chile,
\hfill\break\indent Av.~Vicu\~na Mackenna 4860, Macul, 7820436, Santiago,   Chile.}
\email{{\tt rbenguri@fis.puc.cl} \hfill\break\indent {\it
Web page: }{\tt www.fis.puc.cl/$\sim$rbenguri/}}

\address{Marcone C. Pereira
\hfill\break\indent Dpto. de Matem{\'a}tica Aplicada, IME,
Universidade de S\~ao Paulo, \hfill\break\indent Rua do Mat\~ao 1010, 
S\~ao Paulo - SP, Brazil. } \email{{\tt marcone@ime.usp.br} \hfill\break\indent {\it
Web page: }{\tt www.ime.usp.br/$\sim$marcone}}

\address{Mariel S\'aez
\hfill\break\indent Facultad de Matem\'atica, P. Universidad Cat\'olica de Chile,
\hfill\break\indent Av. Vicu\~na Mackenna 4860, 690444 Santiago, Chile. } 
\email{{\tt mariel@mat.puc.cl} \hfill\break\indent {\it
Web page: }{\tt http://www.mat.uc.cl/$\sim$mariel/}}


\begin{abstract} 
In this paper we obtain a Hadamard type formula for simple eigenvalues and an analog to the Rayleigh-Faber-Krahn inequality for a class of nonlocal eigenvalue problems. 
Such class of equations include among others, the classical nonlocal problems with Dirichlet and Neumann conditions. The Hadamard formula is computed allowing domain perturbations given by embeddings of $n$-dimensional Riemannian manifolds (possibly with boundary) of finite volume while the Rayleigh-Faber-Krahn inequality is shown by rearrangement techniques. 
\end{abstract}

\maketitle

\section{Introduction}
\label{Sect.intro}
\setcounter{equation}{0}

There are many works in the literature which connect the shape of a region to the eigenvalues and eigenfunctions of a given operator.  
For instance, the minimization of eigenvalues (or combination of them) has attracted a lot of attention since the early part of the twentieth century. 
As far as we know, this issue first came out in the famous book of Rayleigh entittled \emph{The theory of sound} \cite{Ray} where he conjectured that 
the disk should minimize the first Dirichlet eigenvalue of the Laplacian among all open sets of given measure.

In the 1920's, Rayleigh's conjecture was simultaneously proved by Faber \cite{Faber} and Krahn \cite{Krahn} using rearrangement techniques. 
Such result is called Rayleigh-Faber-Krahn inequality being one of the most famous isoperimetric inequalities. Naturally similar questions have been
investigated for other eigenvalues as well as for other operators. For instance, we can mention the Payne-P\'olya-Weinberger isoperimetric inequality for the
quotient of the first two Dirichlet eigenvalues of the Laplacian \cite{AsB}; the Szeg\"o-Weinberger inequality, which is an isoperimetric inequality for the 
first nontrivial Neumann eigenvalue of the Laplacian \cite{SW1, SW2}; and \cite{AsB2} where the Rayleigh's conjecture has been considered for the clamped plate. 
For other examples and a more complete bibliography about these issues, we refer to the following surveys \cite{Beng,Henrot0,Henrot}.

Notice that the importance of such kinds of results in analysis, calculus of variations and applied mathematics is self-evident. 
Therefore, the development of more general techniques and approaches to deal with the optimization of functions depending on
the shape of the domains and the eigenvalues of a given operator are required.   
In this context, the rate of change of simple eigenvalues play an essential role and it has been studied since the pioneering work of 
Hadamard \cite{Hada} who first computed the domain derivative of a simple Laplace eigenvalue under Dirichlet boundary conditions in 1908.

Since then, the Hadamard formula has been generalized in a number of significant ways. Such generalizations include the use of 
Neumann and Robin boundary conditions, multiple eigenvalues, and second order variations for a large class of differential and integral operators. 
Among many references, we cite the monographs \cite{Henrot3, Henry, Jimbo, LP} and the recent works \cite{NM, PG, ALau, MPer}, 
all of them concerned with boundary perturbation problems to differential equations and their applications to eigenvalue problems. 
In particular, we mention \cite{Oz} where a proof of the Rayleigh-Faber-Krahn inequality is obtained as a consequence of the analysis of the 
Hadamard formula for the first eigenvalue.

In this work, we study a class of nonlocal eigenvalue problems with non-singular kernels in $n$-dimensional Riemannian manifolds of finite volume.
Our main goal is two fold. We obtain a Hadamard type formula for simple eigenvalues, and an analog of the Rayleigh-Faber-Krahn inequality.
The Hadamard formula is computed in Section \ref{deriS} using the approach developed in \cite{Henry} to deal with boundary perturbation problems. 
The analog of the Rayleigh-Faber-Krahn inequality is obtained in Section \ref{iso} for open bounded sets in $\R^n$ as a direct consequence of 
the rearrangement techniques and the Riesz rearrangement inequality shown respectively in \cite{Beng} and \cite{WB}.

\section{Our nonlocal eigenvalue problem} \label{pre}

We consider an  $n$-dimensional Riemannian manifold $(\Mscr ,g)$ (possibly with boundary) of finite volume and the following nonlocal eigenvalue problem 
\begin{equation} \label{eigeneq}
\begin{gathered}
a_\Mscr (x) u(x) - \int_\Mscr J(x,y) u(y) dy = \lambda(\Mscr ) \, u(x), \quad x \in \Mscr \end{gathered} 
\end{equation}
for some unknown value $\lambda(\Mscr)$ where $a_\Mscr : \overline \Mscr \mapsto \R$ is assumed to be a continuous function, and $J$ is a non-singular kernel satisfying 
$$
{\bf (H)} \qquad 
\begin{gathered}
J \in \mathcal{C}(\Mscr \times \Mscr, \R) \textrm{ is a nonnegative, symmetric function ($J(x,y)=J(y,x)$) }  \\
\textrm{ with } J(x,x)>0.
\end{gathered}
$$
 We also assume that $\int_\Mscr J(x,y)dy<\infty$.
%\begin{remark}
%Here I am not sure what is the right condition for $J$, we want something like
%$$\int_{M} J(d(x,y)) \, dx = 1 \hbox{ for every $x$ }.$$
%but probably is not even true for a general manifold that the integral is independent of $y$.
%\end{remark}
\begin{remark}
Here, $dy$ refers to the measure on the manifold, which in coordinates is equivalent to $\sqrt{g}(x) dx$ and 
$g$ is the determinant of the matrix $g_{ij}$.
\end{remark}

Notice that analysing the spectral properties of \eqref{eigeneq} is equivalent to study the spectrum of 
the linear operator $\mathcal{B}_\Mscr : L^2(\Mscr) \mapsto L^2(\Mscr)$ given by 
\begin{equation} \label{Bdef}
\mathcal{B}_\Mscr u(x) = a_\Mscr(x) u(x) - \int_\Mscr J(x,y) u(y)\, dy, \quad x \in \Mscr.
\end{equation}
See that $\mathcal{B}_\Mscr$ is the difference of the multiplication operator $a_\Mscr $, which maps 
$u(x) \mapsto a_\Mscr(x) u(x)$,
and the integral operator $\mathcal{J}_{\Mscr}: L^2(\Mscr) \mapsto L^2(\Mscr)$ given by  
\begin{equation} \label{opJ}
\mathcal{J}_\Mscr u (x) = \int_\Mscr J(x,y) u(y) dy, \quad x \in \Mscr
\end{equation}
which is self-adjoint and compact by \cite[Propositions 3.5 and 3.7]{AS} since $\Mscr$ is a measurable metric space.

The prototype of the nonlocal equation given by $\mathcal{B}_\Mscr$ is the Dirichlet problem which is defined by $a_\Mscr(x) \equiv 1$. 
For instance, if one takes $\Mscr =\Omega\subset \R^n$, $J(x,y) = J(|x-y|)$, and assume $u(x) \equiv 0$ in $\R^n \setminus \Omega$ 
with $\int_{\R^n} J(z) dz = 1$, the operator $\mathcal{B}_\Omega$ becomes 
$$
\mathcal{B}_\Omega u(x) = \int_{\R^n} J(|x-y|) (u(x) - u(y)) \, dy, \quad x \in \Omega.
$$
Notice that in this case, $\Omega^c = \R^n \setminus \Omega$ is a hostile surrounding since the particles (whose density is set by $u$) die when they land in $\Omega^c$.
As observed, for instance in \cite{ElLibro, Fife}, this is a nonlocal analog to the Laplace operator with Dirichlet boundary condition in bounded open sets of $\R^n$.

On the other hand, if we take $a_\Omega(x) = \int_\Omega J(|x-y|) \, dy$, 
we get a nonlocal analog to the Laplacian with Neumann boundary condition
$$
\mathcal{B}_\Omega u(x) = \int_{\Omega} J(|x-y|) (u(x) - u(y)) \, dy, \quad x \in \Omega.
$$
In this case the particle can just jump inside of $\Omega$ living in an isolated surrounding. 
As expected, under this Neumann condition, the constant function satisfies equation \eqref{eigeneq} whenever one takes $\lambda(\Omega) = 0$.  

According to \cite{LCW}, nonlocal diffusion equations such as \eqref{eigeneq} were used early in population genetics models, 
see for instance \cite{Kimura}. In Ecology, Othmer et~al. \cite{Oth} were the first authors to introduce a jump process to model the dispersion of individuals, 
which later, was generalised by Hutson et~al. \cite{HMMV} associating $J$ to a radial probability density. 

Actually, several continuous models for species and human mobility have been proposed using such 
nonlocal equations, in order to look for more realistic dispersion models \cite{XF, BC, YY, LMS}.
Besides the applied models with such kernels, the mathematical interest is mainly due to the fact that, in general, 
there is no regularizing effect and therefore no general compactness tools are available.

\subsection*{The spectrum of $\mathcal{B}_\Mscr$}

It is known from \cite[Theorem 3.24]{AS} (see also \cite[Theorem 2.2]{LCW} for open bounded sets $\Mscr=\Omega \subset \R^n$) 
that the spectrum set $\sigma(\mathcal{B}_\Mscr)$ of $\mathcal{B}_\Mscr$ satisfies 
\begin{equation} \label{eigenexp}
\sigma(\mathcal{B}_\Mscr) = {\rm R}(a_\Mscr I) \cup \{ \lambda_n(\Mscr) \}_{n=0}^l 
\end{equation}
for some $l \in \{0, 1, ..., \infty \}$ where ${\rm R}(a_\Mscr I)$ denotes de range of the map $a_\Mscr I$ and $\lambda_n(\Mscr)$ are the eigenvalues of $\mathcal{B}_\Mscr$ with finite multiplicity.  
Also, the essential spectrum of $\mathcal{B}_\Mscr$ is given by
$$\sigma_{ess}(\mathcal{B}_\Mscr) = [m, M]$$ 
where
$$
m = \min_{x \in \overline \Mscr} a_\Mscr(x) \quad \textrm{ and } \quad M = \max_{x \in \overline \Mscr} a_\Mscr(x).
$$

As a consequence of the characterization \eqref{eigenexp}, we notice that the eigenfunctions of $\mathcal{B}_\Mscr$ 
possess the same regularity of the functions $J$ and $a_\Mscr$. In fact, for all $x \in \Mscr$, one has 
\begin{equation} \label{eq168}
\mathcal{B}_\Mscr u(x) = \lambda(\Mscr) u(x) \iff ( a_\Mscr(x) - \lambda(\Mscr) ) u(x) = \int_\Mscr J(x,y) u(y) \, dy. 
\end{equation}
On the other hand,  the convolution-type operator  $(J * u)(x) = \int_\Mscr J(x,y) u(y) \, dy \in \mathcal{C}^k(\overline \Mscr)$ 
whenever $J(\cdot, y) \in \mathcal{C}^k(\Mscr)$ for every $y\in \Mscr$ and $u \in L^1(\Mscr)$.
Therefore, if $\lambda(\Mscr)$ is an eigenvalue of $\mathcal{B}_\Mscr$ with corresponding eigenfunction $u$, 
we obtain from \eqref{eigenexp} that $\lambda(\Mscr) \in [m, M]^c$ implying that $a_\Mscr - \lambda(\Mscr) \neq 0$ in $\Mscr$.
Consequently, we get from \eqref{eq168} that 
$$
u \in \mathcal{C}^k(\overline \Mscr) \quad \textrm{ whenever } \quad J (\cdot, y) \textrm{ and } a_\Omega \textrm{ are } \mathcal{C}^k\textrm{-functions}  \textrm{ for every fixed } y\in \Mscr
$$
for $k=0, 1, 2 ...$

Under appropriate conditions, the existence of the principal eigenvalue of $\mathcal{B}_\Mscr$ is guaranteed by \cite[Theorem 2.1]{LCW}.
Recall that the principal eigenvalue of a linear and bounded operator is the minimum of the real part of the spectrum which is simple, isolated and it is associated with
 a continuous and strictly positive eigenfunction.

\section{Hadamard formula for simple eigenvalues} \label{deriS}

Now let us perturb simple eigenvalues of the operator $\mathcal{B}_{\Mscr}$ computing derivatives with respect to several kinds of  variations of 
the manifold $\Mscr$. In the particular case of $\Mscr=\Omega \subset \R^n$ our approach agrees with the one introduced in \cite{Henry} 
for perturbing a fixed domain $\Omega$ by diffeomorphisms. As a consequence, we extend the expressions obtained to the domain derivative for simple eigenvalues given in \cite{RM, GR}.

Let $(\Mscr, g_\Mscr)$ and $(\Nscr,g_\Nscr)$ be  $\mathcal{C}^1$-regular manifolds ($\Mscr,$ possibly with boundary). Assume in addition that $\overline{\Mscr}$ is compact.
If $h:\Mscr \mapsto \Nscr$ is a $\mathcal{C}^1$ imbedding, i.e., a diffeomorphism to its image, we set the composition map $h^*$ (also called the pull-back) by 
$$
h^* v(x) = (v \circ h)(x), \quad x \in \Mscr,
$$
when $v$ is any given function defined on $h(\Mscr)$.  The metric on $\Nscr$ induces the pullback metric on $\Mscr$ through $h$
 as follows:  for $u,\, v\in T_{h^{-1}(x)} \Mscr$ we have $h^* g_\Nscr(u,v)= g_\Nscr(dh_x(u), dh_x(v))$.
It is not difficult to see that $h^*: L^2(h(\Mscr), g_\Nscr  ) \mapsto L^2(\Mscr, h^*g_\Nscr )$ is an isomorphism with inverse 
$(h^*)^{-1} = (h^{-1})^*$.

%In what follows, we will assume that $\Mscr\subset \Nscr$ and that $ g_\Mscr$ is the metric induced by the inclusion $i_\Mscr$,  or equivalently, $ g_\Mscr= i^*_\Mscr g_\Nscr.$
We assume that $\Nscr$ has a Riemannian metric $g_\Nscr$ and we denote by $g_h=h^*g_\Nscr $ the metric on $\Mscr $ 
induced by the embedding $h$.  For instance, if $\Mscr =\Omega\subset \R^n=\Nscr $ then the metric  $h^*g_\Nscr$  
is given by $g_{ij}=\frac{\partial h}{\partial x_i} \cdot \frac{\partial h}{\partial x_j}$ and in particular, if $h=id_{\R^n}$ in 
the  interior of $\Omega$ the metrics in $\Mscr $ and $\Nscr$ agree in that set.

In general, for any embedding $h$ we can consider  the operator  
\begin{equation} \label{Jh}
\begin{gathered}
\left( \mathcal{B}_{h(\Mscr)} v \right)(y)  =  (a_{h(\Mscr)} \circ h)(x) (v \circ h)(x) - \left( \mathcal{J}_{h(\Mscr)} v \right) (h(x)) \\
= ( h^* a_{h(\Mscr)})(x) \, (h^* v)(x) - \left( h^* \mathcal{J}_{h(\Mscr)} v \right) (x)
\end{gathered}
\end{equation}
if $y = h(x)$ for $x \in \Mscr$, where $a_{h(\Mscr)}: h(\Mscr) \mapsto \R$ is assumed to be a continuous function in 
$\overline{h(\Mscr)}$ for any isomorphism $h$. Notice that $\mathcal{B}_{h(\Mscr)}: L^2(h(\Mscr), g_\Nscr) \mapsto L^2(h(\Mscr), g_\Nscr)$ 
is a self-adjoint operator for any $h$ as is  the operator $\mathcal{J}_{h(\Mscr)}$.

On the other hand, we can use the pull-back operator $h^*$ to consider 
$$h^* \mathcal{B}_{h(\Mscr)} {h^*}^{-1}: L^2(\Mscr) \mapsto L^2(\Mscr)$$ 
defined by $h^* \mathcal{B}_{h(\Mscr)} {h^*}^{-1} u(x) = \mathcal{B}_{h(M)} \left( u \circ h \right) \left( h^{-1}(x) \right)$. Hence,
\begin{equation} \label{hJh}
h^* \mathcal{B}_{h(\Mscr)} {h^*}^{-1} u(x) = ( h^* a_{h(\Mscr)})(x) \, u(x) - \left( h^* \mathcal{J}_{h(\Mscr)} {h^*}^{-1} u \right) (x) , \quad \forall x \in \Mscr.
\end{equation}

As it is known, expressions \eqref{Jh} and \eqref{hJh} are the customary way to describe deformations or motions of regions. 
Equation \eqref{Jh} is called the Lagrangian description, and \eqref{hJh} the Eulerian one. 
The latter is written in fixed coordinates while the Lagrangian is not. 

Due to \eqref{Jh} and \eqref{hJh}, it is easy to see that 
\begin{equation} \label{eq820}
\begin{gathered}
h^* \mathcal{B}_{h(\Mscr)} {h^*}^{-1} u(x) = \mathcal{B}_{h(\Mscr)} v(y) \quad \textrm{ and } \quad
h^* \mathcal{J}_{h(\Mscr)} {h^*}^{-1} u(x) = \mathcal{J}_{h(\Mscr)} v(y)
\end{gathered}
\end{equation}
whenever $y=h(x)$ and $v(y) = (u \circ h^{-1})(y) = {h^*}^{-1} u(y)$ for $y \in h(\Mscr)$.

Moreover, we have $\mathcal{B}_{h(\Mscr)} v(y) = \lambda v(y)$ for $y \in h(\Mscr)$ and some value $\lambda$, if and only if, 
$$h^* \mathcal{B}_{h(\Mscr)} {h^*}^{-1} u(x) = \lambda u(x), \quad \forall x \in \Mscr,$$ 
with $v(y) = {h^*}^{-1} u(y)$. Hence, as $\mathcal{B}_{h(\Mscr)}$ 
is a self-adjoint operator for any imbedding $h$, we obtain that the spectrum of $h^* \mathcal{B}_{h(\Mscr)} {h^*}^{-1}$ is also a subset of the real line. We have the following: 
\begin{prop} \label{peigenh}
Let $h:\Mscr \mapsto \Nscr$ be an imbedding. 

Then, $\sigma\left(h^* \mathcal{B}_{h(\Mscr)} {h^*}^{-1}\right) = \sigma\left(\mathcal{B}_{h(\Mscr)}\right) \subset \R$ where $\sigma\left(\mathcal{B}_{h(\Mscr)}\right)$ is given by \eqref{eigenexp}. 
More precisely, $\lambda \in \R$ is an eigenvalue of $\mathcal{B}_{h(\Mscr)}$, if and only if, is an eigenvalue of $h^* \mathcal{B}_{h(\Mscr)} {h^*}^{-1}$. Also, 
$$
\sigma_{ess}\left(h^* \mathcal{B}_{h(\Mscr)} {h^*}^{-1}\right) = \sigma_{ess} \left(\mathcal{B}_{h(\Mscr)}\right).
$$
\end{prop}
\begin{proof}
 As $\mathcal{B}_{h(\Mscr)}$ is a self-adjoint operator in $L^2(h(\Mscr), g_h) $, 
 we have that $\sigma\left(\mathcal{B}_{h(\Mscr)}\right) \subset \R$. Also, we know from relationship \eqref{eq820} 
 that a value $\lambda$ is an eigenvalue of $\mathcal{B}_{h(\Mscr)}$, if and only if, is an eigenvalue of 
$h^* \mathcal{B}_{h(\Mscr)} {h^*}^{-1}$. 
Thus, $\sigma\left(h^* \mathcal{B}_{h(\Mscr)} {h^*}^{-1}\right) = \sigma\left(\mathcal{B}_{h(\Mscr)}\right) \subset \R$ with $\sigma\left(\mathcal{B}_{h(\Mscr}\right)$ given by \eqref{eigenexp}. 

Now, it follows from \eqref{eigenexp} and expressions \eqref{Bdef} and \eqref{hJh} that 
$$
\sigma_{ess} \left(\mathcal{B}_{h(\Mscr)}\right) = [ m_{h}, M_{h} ]
\quad \textrm{ and } \quad 
\sigma_{ess}\left(h^* \mathcal{B}_{h(\Mscr)} {h^*}^{-1}\right) = [ m_{h^*}, M_{h^*} ]
$$
where
$$
m_h = \min_{y \in \overline{h(\Mscr)}} a_{h(\Mscr)}(y), \quad M_h = \max_{y \in \overline{h(\Mscr)}} a_{h(\Mscr)}(y)
$$
and 
$$
m_{h^*} = \min_{x \in \overline{\Mscr}} a_{h(\Mscr)}(h(x)), \quad M_{h^*} = \max_{x \in \overline{\Mscr}} a_{h(\Mscr)}(h(x)).
$$
As $m_h = m_{h^*}$ and $M_h = M_{h^*}$, the proof is completed. 
\end{proof}

\begin{remark}
Notice that Proposition \ref{peigenh} guarantees that the essential spectrum of $\mathcal{B}_{h(\Mscr)}$ does not change under perturbations given by embeddings $h:\Mscr \mapsto \Nscr$.
\end{remark}

From now on, we consider a family of embeddings $h: \Mscr \times[0,T] \mapsto \Nscr$ 
that depends on a parameter $t$. We  denote the perturbed domain $h(\Mscr, t)$ by $\Mscr_t$ in order to simplify the notation.
We study the differentiability of simple eigenvalues $\lambda(\Mscr_t)$ for $\mathcal{B}_{\Mscr_t}$ with respect to $t$. This corresponds to the
 G\^ateaux derivative with respect to the function $h$. 
 
 We remark that for a function $f:\Nscr\to \R $ it holds that 
  $$\frac{d}{dt}\left(h^*f( x, t)\right)=\frac{d}{dt}\left(f(h(x,t), t)\right)=\langle h^* \nabla f , \frac{\partial h }{\partial t}\rangle+h^* \frac{\partial f }{\partial t},$$
where $\nabla$ denotes de tangential gradient on $\Nscr$. 
Then we denote
\begin{equation} \label{Dt}
D_t = \frac{\partial}{\partial t} - \langle   \frac{\partial h}{\partial t} , \nabla \rangle,
\end{equation} where $ \langle \cdot, \cdot \rangle$ denotes the inner product in $\Nscr$.

If $\Nscr=\R^n$ and $\Mscr=\Omega\subset \R^n$ this quantity can be written in coordinates as 
$$D_t = \frac{\partial}{\partial t} - U(t,x) \cdot \frac{\partial}{\partial x} \quad 
\textrm{ with } \quad U(t,x) = {\frac{\partial h}{\partial x}}^{-1} \frac{\partial h}{\partial t} \; \textrm{ for } x \in \Omega$$
and it is known as the anti-convective derivative $D_t$ in the reference domain $\Omega$.

We denote by $\rm{Diff}^1(\Mscr ) \subset \mathcal{C}^1(\Mscr , \Nscr) $ 
the set of $\mathcal{C}^1$-functions $h:\Mscr  \mapsto \Nscr $ which are embeddings. 
We assume that $\Nscr$ has a Riemannian metric $g_\Nscr$ and we denote by $g_h=h^*g_\Nscr $ the metric on $\Mscr $ 
induced by the embedding $h$.  For instance, if $\Mscr =\Omega\subset \R^n=\Nscr $ then the metric  $h^*g_\Nscr$  
is given by $g_{ij}=\langle \partial_{x_i} h, \partial_{x_j} h \rangle$ and, if the dimension of $\Omega$ is $n$, 
the tangent spaces of $h(\Omega)$ and $\R^n$ agree in the interior and   the volume element  is $|Dh| \,dx$.

Consider   the map 
\begin{eqnarray*}
F: {\rm Diff}^1(\Mscr ) \times \R \times L^2(\Mscr ) & \mapsto & L^2(\Mscr) \times \R \\
(h,\lambda, u) & \mapsto & \left(  \left(h^* \mathcal{B}_{h(\Mscr )} {h^*}^{-1}  - \lambda \right) u, \int_\Mscr  u^2(x) dv_{g_h }  \right).
\end{eqnarray*} 

Here $dv_{g_h}$ is the volume element of the metric on $g_h $.
It is not difficult to see that $\rm{Diff}^1(\Mscr)$ is an open set of $\mathcal{C}^1(\Mscr, \Nscr)$ 
(which denotes the space of $\mathcal{C}^1$-functions from $\Mscr$ into $\Nscr$ whose derivatives extend continuously to the closure $\overline \Mscr$ with the usual supremum norm). 
Hence, $F$ can be seen as a map defined between Banach spaces.

In what follows we will consider that $\Mscr \subset \Nscr $ (perhaps by identifying $\Mscr$ with its image with an initial fixed embedding).

Notice that if $\lambda_0 \in \R$ is an eigenvalue for $\mathcal{B}_{\Mscr}$ for some $u_0 \in L^2(\Mscr)$ with $\int_\Mscr u_0^2(x) \, dx = 1$, 
then $F(i_\Mscr, \lambda_0, u_0) = (0,1)$ where $i_\Mscr \in {\rm Diff}^1(\Mscr)$ denotes the inclusion map of $\Mscr$ into $\Nscr$.
On the other hand, whenever $F(h, \lambda, u) = (0,1)$, we have from Proposition \ref{peigenh} that 
$$
\mathcal{B}_{\Mscr_h} v(y) = \lambda v(y), \quad y \in \Mscr_h, 
\quad \textrm{ with } \quad 
\int_{\Mscr_h} v^2(y) \, dy =1
$$
where $v(y) = (u \circ h^{-1})(y)$ for $y \in \Mscr_h$.
In this way, we can proceed as in \cite{RM,Henry} using the map $F$ to deal with eigenvalues and eigenfunctions of 
$\mathcal{B}_{\Mscr_h}$ and $h^* \mathcal{B}_{\Mscr_h} {h^*}^{-1}$ perturbing the eigenvalue problem to the fixed manifold $\Mscr$ by diffeomorphisms $h$.

\begin{theorem} \label{deri}
Let $\lambda_0$ be a simple eigenvalue of $\mathcal{B}_{\Mscr}$ with corresponding normalized eigenfuction $u_0$ and 
 $J \in \mathcal{C}^1(\Nscr \times \Nscr, \R)$satisfying $(\rm{H})$. 
Also, let us assume that $\Phi: {\rm Diff}^1(\Mscr) \mapsto \mathcal{C}^1(\Mscr)$ given by $\Phi(h)(x) = (h^* a_{\Mscr_h})(x)$, $x \in \Omega$, 
is differentiable as a map defined between Banach spaces. 

Then, there is a neighbourhood $\mathcal{O}$ of the inclusion $i_{\Mscr} \in \mathop{\rm Diff}^1(\Mscr)$, and $\mathcal{C}^1$-functions 
$u_h$ and $\lambda_h$ from $\mathcal{O}$ into $L^2(\Mscr)$ and $\mathbb{R}$ respectively satisfying for all $h \in \mathcal{O}$ that 
\begin{equation} \label{eq260}
h^* \mathcal{B}_{h(\Mscr)} {h^*}^{-1} u_h(x) = \lambda_h u_h(x), \quad x \in \Mscr, 
\end{equation}
with $u_h \in \mathcal{C}^1(\Mscr)$.  

Moreover, $\lambda_h$ is a simple eigenvalue with $(\lambda_{i_{\Mscr}}, u_{i_\Mscr} ) = ( \lambda_0, u_0)$ and the domain derivative 
\begin{equation} \label{dode} \begin{gathered}
\frac{\partial \lambda}{\partial h}(i_\Mscr) V = - \int_{\partial \Mscr} \left(  a_\Mscr(s) - \lambda_0  \right) u_0^2(s) \; \langle V^T , N\rangle (s) \, dS 
+ \int_\Mscr u_0^2(x) \, D^T_t (h^*a_{\Omega_h}) \big|_{t=0}  dx  \\ +   \int_{\Mscr} (\lambda_0-a_0)  u_0^2 (w)  \langle \vec{H}, V^{\perp}\rangle dv_0(w)- \int_{\Mscr} u^2_0(w)  \langle\nabla_{w} a_0(w),  V^{\perp}\rangle dv_0(w),
\end{gathered}\end{equation}
for all $V \in \mathcal{X}^1(\Nscr)$ where $\mathcal{X}^1(\Nscr)$ denotes the set of $\mathcal{C}^1$ vector fields on $\Nscr$ and $D_t^Tf=\frac{\partial f}{\partial t}-\langle V^T, \nabla f\rangle$, $V^T$ is the component of $V$ tangential to $\Mscr$. Note that at the boundary the tangent space of $\Mscr$ splits into vectors that are tangential to $\partial M$ (and to $\Mscr$) and one vector that is normal to $\partial \Mscr$ (and tangential to $\Mscr$).  Then  $N\in T(\Mscr)$ denotes this unitary normal vector that is normal to  $\partial M$.  $\vec{H}$ is the mean curvature vector associated to $\Mscr$ and  $V^\perp$ is the component of $V$ normal to $\Mscr$ .
\end{theorem}

%{\color{red} Aca mi idea era que se entendiera que pueden haber en principio muchas direcciones normales (cuando la codimension de $\mathcal{M}$ es grande), pero la normal que aparece aca se entiende bien cual es la direccion en cuestion. Se entiende eso?}

\begin{proof}
The proof of the existence of the neighbourhood $\mathcal{O} \subset \mathop{\rm Diff}^1(\Mscr)$ and the 
$\mathcal{C}^1$-functions $u_h$ and $\lambda_h$ satisfying \eqref{eq260} is very similar to that one performed in \cite[Lemma 4.1]{RM}. 
As one can see, it is a consequence of the Implicit Function Theorem applied to the map $F$. 
Here, we compute the derivative of $\lambda_h$ at $h=i_\Mscr$. 
For this, it is enough to consider a curve of imbeddings $h(t,x)$ that satisfies $h(0,x)=i_\Mscr$ and $\frac{\partial h}{\partial t} =  V(x)$ for a fixed  vector field $V \in \mathcal{X}^1(\Nscr)$. To simplify the notation, we denote the eigenvalue and eigenfunction on $h(t,\Mscr)$ by $\lambda_t$ and $u_t$ respectively. It follows from 
$$
h(t)^* \mathcal{B}_{h(t,\Mscr)} {h(t)^*}^{-1} u_{t}(x) = \lambda_{t} u_{t}, \quad x \in \Mscr,
$$
that 
\begin{equation} \label{eq960}
\frac{\partial}{\partial t} \left( h(t)^* \mathcal{B}_{h(t,\Mscr)} {h(t)^*}^{-1} u_{t}(x) \right)  \Big|_{t=0}
 = \frac{\partial \lambda_{t}}{\partial t} \Big|_{t=0} u_0  + \lambda_0 \frac{\partial u_{t}}{\partial t} \Big|_{t=0} \quad \textrm{ in } \Mscr.
\end{equation}

Now, we need to compute the derivative of the left-hand side of \eqref{eq960}.
Notice that  
\begin{equation} \label{eq333}
\begin{gathered}
\frac{\partial}{\partial t}  \left( h(t)^* \mathcal{B}_{h(t,\Mscr)} {h(t)^*}^{-1} u_{h(t)} \right) \Big|_{t=0} 
 =  \frac{\partial}{\partial t} \left( h^* a_{h(t,\Mscr)} u_{t}\right) \Big|_{t=0}  \\
\qquad \qquad \qquad \qquad \qquad - \frac{\partial}{\partial t} \left( h(t)^* \mathcal{J}_{h(t,\Mscr)} {h(t)^*}^{-1} u_{t} \right) \Big|_{t=0}
\quad \textrm{ in } \Mscr.
\end{gathered}
\end{equation}

Also, for any function $w:\Mscr\times[0,T) \to \R$ it holds that 
$\frac{\partial}{\partial t} \left( h^* w\right)=  h^*\frac{\partial}{\partial t} w+\langle  h^*\nabla w,   \frac{\partial h}{\partial t} \rangle$. Here  $\nabla$ 
denotes de tangential gradient on $\Nscr$   and $\langle \cdot, \cdot \rangle$ the inner product in $\Nscr$. Then we have

\begin{equation} \label{eq1100}
D_t  \left( h(t)^* \mathcal{J}_{h(t,\Mscr)} {h(t)^*}^{-1} u_{t} \right) 
= h(t)^*\frac{\partial}{\partial t} \left(\mathcal{J}_{h(t,\Mscr)} {h(t)^*}^{-1} u_{t} \right)
\quad \textrm{ in } \Nscr.
\end{equation}
In the case of domains of $\R^n$ this derivative is known from \cite[Lemma 4.1]{RM} under the Dirichlet condition. 
%It is obtained using \cite[Lemma 2.1]{Henry}. 
%

Hence, setting $v(t,y) = {h(t)^*}^{-1} u_{t}(y) = u_{t}(h^{-1}(t,y))$, $y \in h(t,\Mscr)$,
we get from \eqref{eq820} 
\begin{eqnarray*}
\frac{\partial}{\partial t} \left(\mathcal{J}_{h(t,\Mscr)} {h(t)^*}^{-1} u_{t} \right) \Big|_{t=0} 
& = & \frac{\partial}{\partial t} \left( \mathcal{J}_{h(t,\Mscr)} v \right) \Big|_{t=0} \\
& = & \frac{\partial}{\partial t} \left( \int_{h(t,\Mscr)} J(y , w) v(t, w) dv_g \right) \Big|_{t=0} \quad \textrm{ for } y \in h(t,\Mscr).
\end{eqnarray*}

To explicitly compute this derivative we recall that  $\left.\frac{dv_g}{dt}\right|_{t=0}=\textrm{tr } (g^{-1} \frac{dg}{dt} (0))\,  dv_0$. 
Since  $g_{ij}(t)=\langle \partial_{x_i} h, \partial_{x_j} h \rangle$ we have that  
$\frac{dg_{ij}}{dt}= \langle \partial_{x_i} h, \partial_{x_j} V\rangle+ \langle \partial_{x_i} V, \partial_{x_j} h\rangle$. Now we denote $V=V^{\perp}+V^T$, 
where $V^{\perp}$ is normal component of $V$ and $V^T$ the tangential one. Then we have  
$$
\left.\frac{dv_g}{dt}\right|_{t=0}=(\textrm{ div}_\Mscr \, V^T +\langle \vec{H}, V^{\perp}\rangle) \, dv_0,
$$ 
where $ \vec{H}$ is the mean curvature 
vector associated to $\Mscr$. To keep in mind the variable that we are using in the computation, we will add a subscript to $\nabla$ (e.g. $\nabla_w J(w,x)$ or $\nabla_x J(w,x)$). Then 
\begin{eqnarray*}
 \frac{\partial}{\partial t} \left( \int_{h(t,\Mscr)} J(y , w) v(t, w) \, dv_g(w) \right) \Big|_{t=0} & = &\int_{\Mscr}  \textrm{ div}_\Mscr \,( J(y , w) v(0, w)  V^T) \, dv_0(w)\\
 &&+ \int_{\Mscr} J(y , w) \frac{d}{dt}v(0, w)\,  dv_0(w) +  \int_{\Mscr} J(y , w) v(t, w) \langle \vec{H}, V^{\perp}\rangle \, dv_0(w)\\ 
 &&+ \int_{\Mscr}  \langle \nabla_{w}( J(y , w) v(0, w)),  V^{\perp} \rangle) \, dv_0(w) \\
  \\&=& \int_{\partial \Mscr}  J(y , w) u_0 (w) \langle V^{T}, N \rangle \,  dS + \int_{\Mscr} J(y , w) D_t u \, dv_0(w)  \\ 
  && +  \int_{\Mscr} J(y , w) u_0  \langle \vec{H}, V^{\perp}\rangle \, dv_0(w)+ \int_{\Mscr}  \langle \nabla_{w}(J(y , w) u_0( w)),  V^{\perp}\rangle  \, dv_0(w),
\end{eqnarray*}
where $N\in T( \Mscr)\cap (T(\partial \Mscr))^{\perp}$ is the unitary normal vector to $\partial M$.  Since $J$ is $\mathcal{C}^1$, the eigenfunctions $u_t$ and their derivatives can be 
continuously extended to the border $\partial \Mscr$.
Hence, $u_t$ possesses trace and the expression above is well defined. Since $u_0$ is a function defined on $\Mscr$ we have 
$\nabla_{\Nscr} u_0( w)= \nabla_{\Mscr} u_0( w)$ and, $\nabla_{\Mscr}u_0( w)$ is tangential to $\Mscr$, then   $\langle\nabla_{\Nscr} u_0( w),  V^{\perp}\rangle=0$. 
We will also denote by $a_0=a_{h(0,\Mscr)}=a_\Mscr$.

%Thus, due to \cite[Theorem 1.11]{Henry}, we can obtain that 
%\begin{eqnarray*}
%& & \frac{\partial}{\partial t} \left(\mathcal{J}_{h(t,\Omega)} {h(t)^*}^{-1} u_{h(t)} \right) \Big|_{t=0} 
% =   \\ 
%& & \qquad \qquad =  \int_\Omega J(x-w) (D_t u|_{t=0}) (0,w) \, dw + \int_{\partial \Omega} J(x-z) u_0(z) \, (V \cdot N_\Omega)(z) \, dS(z)
%\end{eqnarray*}
%where $N_\Omega$ is the unitary normal vector to $\partial \Omega$. 
Consequently, \eqref{Dt} and \eqref{eq1100} imply  
\begin{equation} \label{eq375}
\begin{gathered}
\frac{\partial}{\partial t} \left( h(t)^* \mathcal{J}_{h(t,\Omega)} {h(t)^*}^{-1} u_{t} \right) \Big|_{t=0} = 
\langle V, \nabla \left( \mathcal{J}_{\Omega} u_0 \right)\rangle  + \mathcal{J}_{\Omega} (D_t u |_{t=0}) \\ 
\qquad \qquad \qquad \qquad \qquad 
+ \int_{\partial \Mscr}  J(y , w) u_0 (w) \langle V^{T}, N \rangle \, dS  +  \int_{\Mscr} J(y , w) u_0(w)\,  \langle \vec{H}, V^{\perp}\rangle \, dv_0(w)
  \\ 
\qquad \qquad \qquad \qquad \qquad 
+ \int_{\Mscr}  u_0( w) \, \langle\nabla_{w} (J(y , w)) ,  V^{\perp}\rangle \, dv_0 (w).
 \end{gathered}
\end{equation}

%Since $J$ is $\mathcal{C}^1$, the eigenfunctions $u_h$ and their derivatives can be 
%continuously extended to the border $\partial \Mscr$.
%Hence, $u_h$ possesses trace and expression \eqref{eq375} is well defined. 

We get from \eqref{eq960} and \eqref{eq375} that 
\begin{equation} \label{eq350}
\begin{gathered}
\left. \frac{\partial \lambda_{t}}{\partial t} u_0  + \lambda_0 \frac{\partial u_{t}}{\partial t}\right|_{t=0} = 
\frac{\partial}{\partial t} \left( h^* a_{h(t,\Mscr)} \right) \Big|_{t=0} u_0 + a_0 \frac{\partial u_{t}}{\partial t} 
-\langle V, \nabla \left( \mathcal{J}_{\Mscr} u_0 \right)\rangle  - \mathcal{J}_{\Mscr} (D_t u |_{t=0}) \\
\qquad \qquad 
- \int_{\partial \Mscr}  J(y , w) u_0 (w) \langle V^{T}, N \rangle \, dS  -  \int_{\Mscr} J(y , w) u_0  \langle \vec{H}, V^{\perp}\rangle \, dv_0(w)
  \\ 
\qquad \qquad 
- \int_{\Mscr}u_0(w)\,  \langle\nabla_{w}( J(y , w) ),  V^{\perp}\rangle \, dv_0(w).
\end{gathered}
\end{equation}
Thus, multiplying \eqref{eq350} by the normalized eigenfunction $u_0$ and integrating on $\Mscr$, we obtain 
\begin{equation} 
\begin{gathered}
\frac{\partial \lambda_{t}}{\partial t} + \lambda_0 \int_\Mscr \frac{\partial u_{t}}{\partial t} \, u_0 \, dv_0 (x)= 
\int_\Mscr \frac{\partial}{\partial t} \left( h^* a_{h(t,\Mscr)} \right) \Big|_{t=0} u_0^2 \, dv_0(x)
+ \int_\Mscr \mathcal{J}_{\Mscr} u_0 \frac{\partial u_{t}}{\partial t} \, dv_0(x) \\
+ \int_\Mscr ( a_0 - \mathcal{J}_{\Mscr} u_0) \frac{\partial u_{t}}{\partial t}  \, dv_0(x) 
- \int_ \Mscr   \langle V, \nabla \left( \mathcal{J}_{\Mscr}  \,u_0 \right) \rangle \, u_0 (x) \, dv_0(x)  \\
 - \int_\Mscr u_0 \left[ \mathcal{J}_{\Mscr} (D_t u |_{t=0}) + \int_{\partial \Mscr} J(x,z) u_0(z) \, \langle V^T, N\rangle (z) \, dS(z) \right] \, dv_0(x) \\
 -    \int_{\Mscr}\int_{\Mscr} J(x , w) u_0(x) u_0 (w)  \langle \vec{H}, V^{\perp}\rangle (w) \, dv_0(w) \, dv_0(x)\\
-  \int_{\Mscr} \int_{\Mscr} u_0(x) \, u_0(w)\,  \langle\nabla_{w} J(x, w) ,  V^{\perp}\rangle  (w) \, dv_0(w)\, dv_0(x),
\end{gathered}
\end{equation}
which in turn implies 
\begin{equation} \label{eq360}
\begin{gathered}\left.
\frac{\partial \lambda_{t}}{\partial t}  \right|_{t=0}= 
\int_\Mscr \frac{\partial}{\partial t} \left( h^* a_{h(t,\Mscr )} \right) \Big|_{t=0} \, u_0^2 \, dv_0(x) 
\\+ \int_\Mscr \mathcal{J}_{\Mscr } u_0 \frac{\partial u_{t}}{\partial t} \,  dv_0(x) - \int_ \Mscr  u_0 \, \langle V , \nabla \left( \mathcal{J}_{\Mscr} u_0 \right) \rangle \, dv_0(x)  \\
 - \int_\Mscr u_0(x) \left[ \mathcal{J}_{\Mscr} (D_t u |_{t=0}) + \int_{\partial \Mscr} J(x,y) u_0(z) \, \langle V^T , N \rangle (z) \, dS(z) \right] \, dv_0(x) \\
  -    \int_{\Mscr} \mathcal{J}_\Mscr u_0 (w) u_0 (w) \langle \vec{H}, V^{\perp}\rangle  \, dv_0(w) 
-  \int_{\Mscr} u_0(w)  \langle\nabla_{w}( \mathcal{J}_\Mscr u_0(w) ),  V^{\perp}\rangle \, dv_0(w), 
\end{gathered}
\end{equation}
since $(a_0 - \mathcal{J}_{\Mscr}) u_0 = \lambda_0 u_0$ in $\Mscr$. The last two integrals are obtained from the symmetry  $J(x,w) = J(w,x)$, which also implies
\begin{equation} 
\begin{gathered}
\int_\Mscr u_0 \left[ \mathcal{J}_{\Mscr} (D_t u |_{t=0}) +  \langle V , \nabla_x \left( \mathcal{J}_{\Mscr} u_0 \right) \rangle \right] dv_0(x)  \\
= \int_\Mscr \int_\Mscr J(x, w) u_0(x) \left( \frac{\partial u_{t}}{\partial t}(w) - \langle V(w) , \nabla_w u_0(w)\rangle \right) dv_0(w)  dv_0(x) 
+ \int_\Mscr u_0 \langle V ,\nabla_x \left( a_\Omega u_0 - \lambda_0 u_0  \right)\rangle d v_0(x) \\
= \int_\Mscr \frac{\partial u_{t}}{\partial t} \mathcal{J}_{\Mscr} u_0 \, d v_0(x) - \int_\Mscr (a_0 - \lambda_0) u_0  \langle V , \nabla u_0 \rangle dv_0(x) 
+ \int_\Mscr u_0  \langle V,  \nabla \left( a_\Omega u_0 - \lambda_0 u_0  \right) \rangle dv_0(x) \\
= \int_\Mscr \frac{\partial u_{t}}{\partial t} \mathcal{J}_{\Mscr} u_0 \, dv_0(x) + \int_\Mscr u_0^2 \langle V, \nabla  a_0 \rangle d v_0(x) .
\end{gathered}
\end{equation}
Finally we observe
 \begin{equation}  \label{eq360m}
\begin{gathered}\int_{\Mscr} u_0(w)  \langle\nabla_{w}( \mathcal{J}_\Mscr u_0(w) ),  V^{\perp}\rangle dv_0(w)= \int_{\Mscr} u^2_0(w)  \langle\nabla_{w}( a_0(w) ),  V^{\perp}\rangle dv_0(w). 
\end{gathered}
\end{equation}
Here we used that $\nabla u_0$ is tangential to $\Mscr$.

Consequently, we get from \eqref{eq360} that 
\begin{eqnarray*} 
\frac{\partial \lambda_{t}}{\partial t}(0) & = & 
\int_\Mscr \frac{\partial}{\partial t} \left( h^* a_{h(t,\Mscr )} \right) \Big|_{t=0} \, u_0^2 \, dv_g(x)  - \int_\Mscr u_0^2 \langle V, \nabla  a_0 \rangle d v_0(x)  \\
 & & - \int_\Mscr u_0(x) \int_{\partial \Mscr} J(x,y) u_0(z) \, \langle V^T , N \rangle (z) \, dS(z) \, dv_g(x) \\
& &  -    \int_{\Mscr} \mathcal{J}_\Mscr u_0 (w) u_0 (w) \langle \vec{H}, V^{\perp}\rangle  \, dv_g(w) 
-  \int_{\Mscr} u^2_0(w)  \langle\nabla_{w}( a_0(w) ),  V^{\perp}\rangle dv_g(w) \\
& = & \int_\Mscr u_0^2 \, D^T_t (h^*a_{t}) \big|_{t=0} \, dv_t(x) -  \int_{\partial \Mscr} ( a_0 - \lambda_0) u_0^2 \langle V , N\rangle \, dS \\ &&  -    \int_{\Mscr} ( a_0 - \lambda_0) u_0^2 (w) \langle \vec{H}, V^{\perp}\rangle dv_g(w) - \int_{\Mscr} u^2_0(w)  \langle\nabla_{w} a_0(w) ,  V^{\perp}\rangle \, dv_g(w)
\end{eqnarray*}
where  $D_t^Tf=\frac{\partial f}{\partial t}-\langle V^T, \nabla f\rangle$. Observing that $ \mathcal{J}_\Mscr u_0 = (a_0-\lambda_0) u_0$ we complete the proof. 
\end{proof}

\begin{remark} In the case that $\Mscr$ is a open domain of $\R^{n+1}$ (with co-dimension 0). The formula becomes
$$\frac{\partial \lambda}{\partial h}(i_\Omega)  V = - \int_{\partial \Omega} \left(  a_\Omega(s) - \lambda_0  \right) u_0^2(s) \; \langle V, N \rangle \, dS 
+ \int_\Omega u_0^2(x) \, D_t (h^*a_{\Omega_h}) \big|_{t=0}  dx.$$
\end{remark}

Next, we give some preliminary examples setting suitable nonlocal operators computing their Hadamard formula.

\begin{example}[The sphere]{\rm Consider $\Mscr=\mathbb{S}^n$ and $\Nscr=\mathbb{R}^{n+1}$. If we take $a\equiv 0$, then $\vec{H}(p)=\frac{n}{R}p$ and
$$
\frac{\partial \lambda_{t}}{\partial t}(i_\Omega)  V  =  
  \frac{n \lambda_0}{R}   \int_{\mathbb{S}^n}  u^2_0 (w) \langle w, V^{\perp}\rangle \, dv_0(w).
$$
}
\end{example}

\begin{example}[The Dirichlet problem on the upper hemisphere]
{\rm Consider $\Mscr=\mathbb{S}^n_+$ (that is $p\in \mathbb{S}^n$ with $x_{n+1}\geq 0$) and $\Nscr=\mathbb{R}^{n+1}$. 
If we take $a\equiv 1$, then  $\vec{H}(p)=\frac{1}{R}p$ and $N(p)=e_{n+1}$
$$
\frac{\partial \lambda_{t}}{\partial t}(i_{\mathbb{S}^n_+})  V =
-(1- \lambda_0)  \int_{\partial \mathbb{S}^{n}_+}   u_0^2  V_{n+1} \, dS  - \frac{n (1- \lambda_0) }{R}   \int_{\mathbb{S}^n}  u^2_0 (w) \langle w, V^{\perp}\rangle \, dv_0(w).
$$
}
\end{example}

\begin{example}[One parameter family of functions $a$]{\rm
Consider $\Omega\subset \mathbb{R}^n$, $\Mscr=\Omega\times [0,1]$ and $\Nscr=\mathbb{R}^{n+1}$. In this case $H=0$, but assume that $a$ depends on the variable $x_{n+1}$ then

$$
\begin{gathered}
\frac{\partial \lambda}{\partial h}(i_\Omega)  V = - \int_{\partial \Omega} \left(  a_\Omega(s) - \lambda_0  \right) u_0^2(s) \; (V \cdot N)(s) \, dS 
+ \int_\Omega u_0^2(x) \, D_t (h^*a_{\Omega_h}) \big|_{t=0}  dx \\ 
\qquad \qquad -2 \int_{\Omega\times[0,1]} u^2_0(w)  \langle\nabla_{w} a_0(w) ,  V^{\perp}\rangle \, dv_0(w).
\end{gathered}$$

}
\end{example}

%{\color{red} Aca la idea era ejemplificar que variaciones en la funcion $a$ se podian ver como variaciones en el dominio, tendra sentido elaborar mas al respecto?}

\subsection*{Domain derivative of eigenfunctions}

Let us now determine the domain derivative of the function $u_h$ introduced by Theorem \ref{deri} at the reference manifold $\Mscr$. 

Due to \eqref{eq350}, we have for all $V \in \mathcal{X}^1(\Nscr)$ that
\begin{equation*} 
\begin{gathered}
\frac{\partial \lambda_{i_\Mscr}}{\partial t} u_0  + \lambda_0 \frac{\partial u_{i_\Mscr}}{\partial t} = 
\frac{\partial}{\partial t} \left( h^* a_{h(t,\Mscr)} \right) \Big|_{t=0} u_0 + \mathcal{B}_\Mscr \left( \frac{\partial u_{i_\Mscr}}{\partial t} \right) \\
\qquad \qquad \qquad + \left[ \mathcal{J}_\Mscr, \langle V, \nabla (\cdot) \rangle \right] u_0
- \int_{\partial \Mscr} J(y,w) u_0(w) \langle V, N \rangle \, dS(w)\\
-  \int_{\Mscr} J(y , w) u_0  \langle \vec{H}, V^{\perp}\rangle dv_0(w)
- \int_{\Mscr}u_0(w)\,  \langle\nabla_{w}( J(y , w) ),  V^{\perp}\rangle dv_0(w).
\end{gathered}
\end{equation*}
where $\left[A, B \right] u := ABu - BA u$, and then, $\left[ \mathcal{J}_\Mscr, \langle V, \nabla (\cdot) \rangle \right] u_0 = \mathcal{J}_\Mscr (\langle V, \nabla u_0 \rangle) - \langle V, \nabla (\mathcal{J}_\Mscr u_0) \rangle $. 

Hence,
\begin{equation} \label{eq588}
\begin{gathered}
(\lambda_0 - \mathcal{B}_\Mscr) \frac{\partial u_{i_\Mscr}}{\partial t} = - \frac{\partial \lambda_{i_\Mscr}}{\partial t} u_0 
+ \frac{\partial}{\partial t} \left( h^* a_{h(t,\Mscr)} \right) \Big|_{t=0} u_0 + \left[ \mathcal{J}_\Mscr, \langle V, \nabla (\cdot) \rangle \right] u_0 \\
- \int_{\partial \Mscr} J(y,w) u_0(w) \langle V, N \rangle \, dS(w)
-  \int_{\Mscr} J(y , w) u_0  \langle \vec{H}, V^{\perp}\rangle dv_0(w)\\
- \int_{\Mscr}u_0(w)\,  \langle\nabla_{w}( J(y , w) ),  V^{\perp}\rangle dv_0(w).
\end{gathered}
\end{equation}
Thus, we can conclude that the derivative of $u_h$ at $h=i_\Mscr$ in $V \in \mathcal{C}^1(\Nscr, \Nscr)$ is the solution of 
$$
(\lambda_0 - \mathcal{B}_\Mscr) w = f_V
$$
where $f_V \in L^2(\Mscr)$ is the function given by the right side of \eqref{eq588} which is well defined since $u_0$, $\lambda_0$, $\frac{\partial \lambda_{i_\Mscr}}{\partial t}$ and $\frac{\partial}{\partial t} \left( h^* a_{h(t,\Mscr)} \right) \Big|_{t=0}$ are known.

Notice that $\lambda_0$ is a simple eigenvalue of $\mathcal{B}_\Mscr$, and then, we have
$$
L^2(\Mscr) = {\rm R}(\lambda_0 - \mathcal{B}_{\Mscr}) \oplus [u_0]. 
$$
Therefore, \eqref{eq588} possesses unique solution, if and only if, $\int_\Mscr u_0 f_V dx = 0$ for each $V \in \mathcal{X}^1(\Nscr)$.

Indeed, it follows from \eqref{dode} and the assumption $\mathcal{B}_{\Mscr} u_0 = \lambda_0 u_0$ in $\Mscr$ that 
\begin{eqnarray*}
\int_\Mscr u_0 f_V dx & = & - \frac{\partial \lambda_{t}}{\partial t} 
+ \int_\Mscr D_t( h^* a_{h(t,\cdot)} ) \Big|_{t=0} u_0^2  \, dx + \int_\Mscr u_0  (u_0 \langle V, \nabla (\cdot) \rangle a_\Mscr 
+ \left[ \mathcal{J}_\Mscr, \langle V, \nabla (\cdot) \rangle \right] u_0 ) \, dx \\
& & - \int_\Mscr \int_{\partial \Mscr} J(x,z) u_0(x) \, u_0(z) \, \langle V, N \rangle \, dS(z) dx \\
& &-\int_\Mscr \int_{ \Mscr} J(x,z) u_0(x) \, u_0(z) \,\langle \vec{H}, V^{\perp}\rangle dz\, dx \, \\
& &-\int_\Mscr \int_{ \Mscr}  u_0(x) \, u_0(z) \, \langle (\nabla_x J), V^{\perp}\rangle dz \, dx\,
 \\
& = & \int_\Mscr u_0 \left(  \mathcal{J}_\Mscr(\langle V^T, \nabla u_0 \rangle ) - a_\Mscr \langle V^T, \nabla u_0 \rangle  + \lambda_0 \langle V^T, \nabla u_0 \rangle  \right) dx \\
& = & \int_\Mscr u_0 \left( \lambda_0 - \mathcal{B}_{\Mscr} \right)(\langle V^T, \nabla u_0 \rangle ) \, dx.
\end{eqnarray*}

Thus, since $\mathcal{B}_{\Mscr}$ is a self-adjoint operator, $u_0$ is a $\mathcal{C}^1$-function and $\langle V, \nabla (\cdot) \rangle u_0 \in L^2(\Mscr)$, one has   
$$
\int_\Mscr u_0 f_V dx = \int_\Mscr  \langle V, \nabla u_0 \rangle \left( \lambda_0 - \mathcal{B}_{\Mscr} \right) u_0  \, dx = 0
$$
for all $V \in \mathcal{X}^1(\Nscr)$ which proves the following result.

\begin{cor}
Let $u_h$ be the family of eigenfunctions associated with the operator $\mathcal{B}_{h(\Mscr)}$  and eigenvalues $\lambda_h$ given by Theorem \ref{deri}. 

Then, the derivative of $u_h$ at $h=i_\Mscr$ and $V \in \mathcal{X}^1(\Nscr)$ is the unique solution of 
$$
(\lambda_0 - \mathcal{B}_\Mscr) w = f_V
$$
where $f_V \in L^2(\Mscr)$ is the function given by 
$$
\begin{gathered}
f_V = - \frac{\partial \lambda}{\partial h}(i_\Mscr)  V \, u_0 
+\frac{\partial}{\partial t} \left( h^* a_{h(t,\Mscr)} \right) \Big|_{t=0} u_0 + \left[ \mathcal{J}_\Mscr, \langle V, \nabla (\cdot) \rangle \right] u_0 \\
- \int_{\partial \Mscr} J(y,w) u_0(w) \langle V, N \rangle \, dS(w)\\
-  \int_{\Mscr} J(y , w) u_0  \langle \vec{H}, V^{\perp}\rangle dv_0(w)\\
- \int_{\Mscr}u_0(w)\,  \langle\nabla_{w}( J(y , w) ),  V^{\perp}\rangle dv_0(w)
\end{gathered}
$$
with $\frac{\partial \lambda}{\partial h}(i_\Mscr)  V$ given  by \eqref{dode}.
\end{cor}

%\begin{remark}
%It is worth mentioning that all the results discussed to this point remain valid substituting the radial condition on the function $J$ with the symmetric one, i.e., assuming $J(x,y)=J(y,x)$. As we will see, the radial assumptions on $J$ is just needed to use rearrangement techniques.
%\end{remark}

\subsection*{Some examples in Euclidean spaces}

In the sequel we compute some examples assuming $\Mscr =\Omega$ is an open set in $\R^n$, $J(x,y) = J(|x-y|)$ with $\int_{\R^n} J(z) dz = 1$. 
It is worth noting that such examples often appear in the literature associated with nonlocal equations. Below we give appropriate references for each example considered.

\begin{example}[Dirichlet problem] {\rm 
If we take $a_\Omega(x) \equiv 1$ in \eqref{Bdef}, we have what is called the Dirichlet nonlocal problem.
In this case, the Hadamard formula is known and it was first obtained in \cite{GR} for the first eigenvalue. 
In \cite{RM}, we have proved that the same formula still holds for any simple eigenvalue. 
Since $a_\Omega$ is constant, $D_t (h^*a_{\Omega_h}) \big|_{t=0} = 0$ and, from Theorem \ref{deri}, we get  
$$
\frac{\partial \lambda}{\partial h}(i_\Omega) V = - \left(  1 - \lambda_0  \right) \int_{\partial \Omega} u_0^2 \; V \cdot N \, dS
\quad \forall V \in \mathcal{C}^1(\Omega, \R^n)
$$ }
\end{example}
with $\cdot$ denoting the scalar product in $\R^n$.

\begin{example}[Neumann problem] {\rm 
In the literature, see for instance \cite{ElLibro, Fife, HMMV}, the nonlocal Neumann problem is established taking 
$$
a_\Omega(x) = \int_\Omega J(|x-y|) dy, \quad x \in \R^n.
$$
As expected, zero is its first eigenvalue for any measurable open set $\Omega$ which is simple and it is associated with a constant eigenfunction. 
Clearly, the rate of the first eigenvalue with respect to the domain must be null. Let us take its rate for any other simple eigenvalue.
For this, we first compute the anti-convective derivative of $a_\Omega$ at $t=0$ assuming $h(t,x) = x + t V(x)$ 
for some $V \in \mathcal{C}^1(\Omega, \R^n)$. We have from \cite[Lemma 2.1]{Henry} and \cite[Theorem 1.1]{Henry} that 
$$
\begin{gathered}
D_t \left[ h^*(t) a_{h(t,\Omega)}  \right] \big|_{t=0} 
= h^*(t) \frac{\partial}{\partial t} \left[ \int_{h(t,\Omega)} J(|\cdot - w|) dw \right] \Big|_{t=0} \\
= \int_{\partial \Omega} J(|x-s|) (V \cdot N)(s) \, dS, \quad x \in \Omega. 
\end{gathered}
$$
Hence, we obtain from Theorem \ref{deri} that 
$$
\begin{gathered}
\frac{\partial \lambda}{\partial h}(i_\Omega) V = - \int_{\partial \Omega} \left(  a_\Omega(s) - \lambda_0  \right) u_0^2(s) \; (V \cdot N)(s) \, dS \\
+ \int_\Omega u_0^2(x) \int_{\partial \Omega} J(|x-s|) (V \cdot N)(s) \, dS dx \\
=  - \int_{\partial \Omega} \left(  a_\Omega(s) - \lambda_0  \right) u_0^2(s) \; (V \cdot N)(s) \, dS 
+ \int_{\partial \Omega} (\mathcal{J}_\Omega u_0^2)(s) (V \cdot N)(s) \, dS .
\end{gathered}
$$ 
Notice in the last integral the term $\mathcal{J}_\Omega u_0^2$ which is 
the operator $\mathcal{J}_\Omega$ applied to the square of the normalized eigenfunction $u_0$. 

%Moreover, if we take $\lambda_0$ as the first eigenvalue, we know that $\lambda_0=0$ with corresponding eigenfunction $u_0 = |\Omega|^{-1/2}$ where $|\cdot|$ denotes the Lebesque measure of measurable sets on $\R^n$. Thus, we get as expected that 
%\begin{eqnarray*}
%\frac{\partial \lambda}{\partial h}(i_\Omega) \cdot V & = &
%- \frac{1}{|\Omega|} \int_{\partial \Omega} \left(  \int_\Omega J(x-y) dy  \right) (V \cdot N_\Omega)(s) \, dS \\
%& & \quad  + \frac{1}{|\Omega|} \int_\Omega\int_{\partial \Omega} J(x-s) (V \cdot N_\Omega)(s) \, dS dx \\
%& = & 0 \quad \forall V \in \mathcal{C}^1(\Omega, \R^n), 
%\end{eqnarray*}
%since $J(-x) = J(x)$ in $\R^n$.
}
\end{example}

\begin{example} {\rm 
Let $D \subset \R^n$ be a bounded open set and take $A \subset D$, another open bounded set strictly contained in $D$ 
in such way that $\partial A \cap \partial D = \emptyset$. 
Next, consider $\Omega = D \setminus A$ defining  
$$
a_\Omega(x) = \int_{\R^n \setminus A} J(|x-y|) \, dy, \quad x \in \R^n.
$$
The nonlocal operator $\mathcal{B}_\Omega$ given for such function $a_\Omega$ is a kind of Dirichlet/Neumann problem. 
It takes Dirichlet boundary condition side out of $D$ setting Neumann condition on the hole $A$. 
Such operator is given by 
$$
\mathcal{B}_\Omega(x) = \int_{\R^n \setminus A} J(|x-y|) (u(x) - u(y))\, dy, \quad x \in \Omega,
$$
assuming $u \equiv 0$ in $\R^n \setminus D$ and has been studied for instance in \cite{PR}. 
Let us compute its Hadamard formula. Due to  
$$
\begin{gathered}
a_\Omega(x) = \int_{\R^n} J(|x-y|) dy - \int_A J(|x - y|) \, dy \\
= 1 - \int_D J(|x - y|) \, dy + \int_\Omega J(|x - y|) \, dy, \quad \forall x \in \R^n,
\end{gathered}
$$
one gets again from \cite[Lemma 2.1]{Henry} and \cite[Theorem 1.1]{Henry} that 
$$
\begin{gathered}
D_t \left[ h^*(t) a_{h(t,\Omega)}  \right] \big|_{t=0} 
= h^*(t) \frac{\partial}{\partial t} \left[1 - \int_{h(t,D)} J(|\cdot - y|) \, dy + \int_{h(t,\Omega)} J(|\cdot - y|) \,  dy \right] \Big|_{t=0} \\
= - \int_{\partial D} J(|x-s|) (V \cdot N)(s) \, dS + \int_{\partial \Omega} J(|x-s|) (V \cdot N)(s) \, dS \\
= \int_{\partial A} J(|x-s|) (V \cdot N)(s) \, dS, \quad x \in \Omega
\end{gathered}
$$
since $\partial \Omega = \partial D \cup \partial A$ with $\partial D \cup \partial A = \emptyset$.
Hence,
$$
\begin{gathered}
\frac{\partial \lambda}{\partial h}(i_\Omega)  V = - \int_{\partial \Omega} \left(  a_\Omega(s) - \lambda_0  \right) u_0^2(s) \; (V \cdot N)(s) \, dS \\
+ \int_{\partial A}  (\mathcal{J}_\Omega u_0^2)(s) (V \cdot N)(s) \, dS.
\end{gathered}
$$
}
\end{example}

%{\color{red} Tendr\'a sentido agregar aca los calculos  para la esfera, semi esfera y cilindro que estan en la seccion anterior?}

\section{Isoperimetric inequalities for eigenvalues} \label{iso}

In this section, we obtain an analogue of the Rayleigh-Faber-Krahn inequality for the operator $\mathcal{B}_\Mscr$ 
assuming $\Mscr =\Omega$ is an open set in $\R^n$ and the function $J$ satisfies  
$$
{\bf (H)} \qquad 
\begin{gathered}
J \in \mathcal{C}(\R^n,\R) \textrm{ is a nonnegative function, spherically symmetric and radially decreasing } \\ 
\textrm{ with } J(0)>0 \textrm{ and } 
\int_{\R^n} J(x) \, dx = 1.
\end{gathered}
$$
As we will see, it is a direct consequence of rearrangements (or Schwarz symmetrization) first introduced by Hardy and Littlewood \cite{HL}.
In the following we recall some basic definitions and properties concerning spherically symmetric rearrangements.
We mention \cite{Beng}, as well as \cite{CB, Henrot}, for more detailed discussions and proofs concerning this subject. 

Let $\Omega \subset \R^n$ be a measurable set and $|\Omega|$ its Lebesgue measure. If $|\Omega|$ is finite, 
we denote by $\Omega^*$ an open ball with the same measure as $\Omega$, otherwise, we write $\Omega^* = \R^n$. 
We consider $u:\Omega \mapsto \R$ a measurable function assuming either that $|\Omega|$ is finite or that $u$ decays at infinity, i.e., the set
$|\{ x \in \Omega \; : \; |u(x)| > t \}|$ is finite for all $t>0$.

The function $\mu(t) = |\{ x \in \Omega \; : \; |u(x)|>t \}|$ defined for $ t \geq 0$ is called the distribution function of $u$.
It is non-increasing, right-continuous with $\mu(0) = |{\rm supp}(u)|$ and ${\rm supp}(\mu) = [0, \| u \|_{L^\infty(\Omega)})$. 

The decreasing rearrangement $u^{\#}: \R^+ \mapsto \R^+$ of $u$ is the distributional function of $\mu$, 
and it can be used to set the decreasing symmetric rearrangement $u^*:\Omega \mapsto \R^+$ of $u$ which is defined by 
$$u^*(x) = u^{\#}(c_n |x|^n)$$ 
where the constant $c_n = \pi^{n/2} \left( \Gamma(n/2 +1) \right)^{-1}$ is the measure of the $n$-dimensional unit ball. 

It follows from \cite[Lemma 3.4]{Beng} that $u^*$ is spherically symmetric and radially decreasing. 
Also, the measure of the level set $\{ x \in \Omega^* \; : \; u^*(x) > t \}$ is  the same as the measure of $\{ x \in \Omega \; : \; |u(x)| > t \}$ for any $t \geq 0$. 

Quite analogous to the decreasing rearrangements are the definitions of increasing ones. 
If the measure of $\Omega$ is finite, we set by $u_{\#}(s) = u^{\#}(|\Omega|- s)$ the increasing rearrangement of $u$.
Hence, the symmetric increasing rearrangement $u_*:\Omega^* \mapsto \R^+$ of $u$ is defined by 
$$u_*(x) = u_{\#}(c_n |x|^n).$$

Due to the symmetry condition imposed on the kernel $J$, we can show an analogue of the Rayleigh-Faber-Krahn inequality for the operator 
$\mathcal{B}_\Omega$ assuming $a_\Omega$ is a non-negative function. 
In this way, we improve previous results obtained in \cite{RM, RS} for the Dirichlet nonlocal problem and the compact operator $\mathcal{J}_\Omega$.
We show that the first eigenvalue of $\mathcal{B}_\Omega$ possesses as a lower bound, the first eigenvalue of the following self-adjoint operator: 
$\mathcal{B}^*_{\Omega^*}: L^2(\Omega^*) \mapsto L^2(\Omega^*)$ given by 
\begin{equation} \label{B*}
\mathcal{B}^*_{\Omega^*} u(x)  = {a_*}_{\Omega^*}(x) u(x) - \int_{\Omega^*} J(x-y) u(y) \, dy, \quad x \in \Omega^*
\end{equation}
where the function ${a_*}_{\Omega^*}$ is the symmetric increasing rearrangement of $a_\Omega$.
It is a consequence of the Riesz rearrangement inequality proved in \cite[Symmetrization Lemma]{WB}. It is known that 
\begin{equation} \label{eq236}
\int_{\R^n} \int_{\R^n} f(y) g(y-x) h(x) dx dy \leq \int_{\R^n} \int_{\R^n} f^*(y) g^*(y-x) h^*(x) dx dy
\end{equation}
for any nonnegative measurable functions $f$, $g$ and $h$ defined in $\R^n$.

\begin{theorem} \label{theolamb}
Let $\Omega \subset \R^n$ be an open bounded set, $a_\Omega: \overline{\Omega} \mapsto \R^+$ a non-negative continuous function 
and $\Omega^*$ an open ball with the same measure as $\Omega$. 
Assume that there exist the first eigenvalues of $\mathcal{B}_\Omega$ and $\mathcal{B}^*_{\Omega^*}$ denoted respectively by $\lambda_1(\Omega)$ and $\lambda_1^*(\Omega^*)$. 

Then, under conditions $(\rm{H})$, we have that 
$$
\lambda_1(\Omega) \geq \lambda_1^*(\Omega^*).
$$
\end{theorem}
\begin{proof}
First, let us recall that \cite[Theorem 3.8]{Beng} implies that 
\begin{equation} \label{eq248}
\int_\Omega \phi(x) \varphi(x) \, dx \geq \int_{\Omega^*} \phi_*(x) \varphi^*(x) \, dx
\end{equation}
for any non-negative functions $\phi$ and $\varphi$ defined on $\Omega \subset \R^n$. 
Consequently, it follows from \eqref{eq248} and the Riesz rearrangement inequality \eqref{eq236} that 
\begin{eqnarray*}
\int_\Omega u(x) \, (\mathcal{B}_\Omega u)(x) \, dx & = & 
\int_\Omega a_\Omega(x) u^2(x) \, dx - \int_\Omega \int_\Omega J(x-y) u(y) u(x) \, dy dx  \\
& \geq & \int_{\Omega^*} {a_*}_{\Omega^*}(x) {u^*}^2(x) \, dx - \int_{\Omega^*} \int_{\Omega^*} J(x-y) u^*(y) u^*(x) \, dy dx \\
& \geq & \int_{\Omega^*} u^*(x) \, (\mathcal{B}^*_{\Omega^*} u^*)(x) \, dx
\end{eqnarray*}
since $J$ is nonnegative, spherically symmetric and radially decreasing.

Now, due to \cite[Theorem 3.6]{Beng}, we know that 
$$\| u^* \|_{L^2(\Omega^*)} = \| u \|_{L^2(\Omega)}$$ 
for any nonnegative function $u$. 
Thus, if $u_1$ is the corresponding eigenfunction of $\lambda_1(\Omega)$, one has 
$$
\lambda_1(\Omega) \geq \int_{\Omega^*} u_1^*(x) \, (\mathcal{B}^*_{\Omega^*} u_1^*)(x) \, dx \geq \lambda_1^*(\Omega^*)
$$
completing the proof. 
\end{proof}

\begin{remark} {\rm 
We recall that, under appropriate conditions, the existence of the first eigenvalue $\lambda_1(\Omega)$ of $\mathcal{B}_\Omega$ 
is guaranteed by \cite[Theorem 2.1]{LCW}. In particular, $\lambda_1(\Omega)$ exists if $a_\Omega$ satisfies  
\begin{equation} \label{eq272}
\int_\Omega \frac{dx}{a_\Omega(x) - m} = \infty
\end{equation}
with $m = \min_{x \in \overline \Omega} a_\Omega(x)$. 
Now, we known from Proposition \ref{dfun} (which is a consequence of the layer-cake formula) that 
\begin{equation} \label{eq277}
\int_\Omega \Phi(u(x)) \, dx = \int_{\Omega^*} \Phi(u_*(x)) \, dx
\end{equation}
for any non-negative measurable function $u$ and any decreasing function $\Phi$ satisfying 
\begin{equation} \label{eq694}
\lim_{t\to \infty} \Phi(t)=0 \quad \textrm{ and } \quad \lim_{t\to a+} \Phi(t)=\infty.
\end{equation} 
Therefore, since $\Phi(x) = (x-m)^{-1}$ is a non-negative decreasing function on $(m, +\infty)$ satisfying \eqref{eq694}, 
we obtain from \eqref{eq272} and \eqref{eq277} that
$$
\int_{\Omega^*} \frac{dx}{{a_*}_\Omega(x) - m} = \infty.
$$
Thus, it follows from \cite[Theorem 2.1]{LCW} that the first eigenvalue $\lambda^*_{\Omega^*}$ of \eqref{B*} also exists ensuring the application of the 
isoperimetric inequality given by Theorem \ref{theolamb} to a large class of nonlocal operators $\mathcal{B}_\Omega$. }
\end{remark}

Notice that the ball is not the unique minimizer of $\lambda_1(\Omega)$ even up to displacements. Indeed, since $L^2(\Omega)$ does not change 
if we remove from $\Omega$ a set of zero measure, any kind of open sets as $\Omega^* \setminus A$ with $|A|=0$ gives a minimizer for $\lambda_1(\Omega)$.

As we have already mentioned, the operators $\mathcal{B}_\Omega$ and $\mathcal{B}^*_{\Omega^*}$ can be introduced by a jump process 
used to model dispersion of individuals in a given habitat. In fact, if $u(x,t)$ is thought of as a population density at a point $x$ and a time $t$, 
and $J(x-y)$ is the probability distribution of jumping from a location $y$ to the position $x$, the amount $\int_\Omega J(x-y) \, u(y,t) \, dy$ gives 
the rate in which individuals are arriving at location $x$ from all the other places $y \in \Omega$. On the other hand, $- a_\Omega(x) u(x,t)$ 
can be thought of as the rate in which individuals are leaving position $x$ to the others sites in the habitat. Therefore, in the absence of external 
or internal sources, we have that the density $u$ satisfies the evolution equation $u_t(x,t) = - \mathcal{B}_\Omega u(x,t)$, $x \in \Omega$.
Hence, it follows from Theorem \ref{theolamb} that the minimum decay rate of the population density $u(x,t)$ is attained when the habitat is a ball.

Under a Neumann condition, i.e., assuming $a_\Omega(x) = \int_\Omega J(x-y) \,dy$ in the definition of $\mathcal{B}_\Omega$, it is clear that 
zero is the lower bound for the first eigenvalues since it is the principal eigenvalue of $\mathcal{B}_\Omega$ for any bounded open set $\Omega \subset \R^n$. 
Anyway, as we have ${a_*}_{\Omega^*}(x) = a_{\Omega^*}(x)$ in $\Omega^*$, we also recover this obvious property using Theorem \ref{theolamb}.

Finally, we notice that in general, the first eigenvalue $\lambda_1(\Omega)$ of \eqref{eigeneq} does not have a maximizer 
among open bounded sets with constant measure. For the Dirichlet problem, i.e., under the assumption $a_\Omega(x) \equiv 1$ 
for all $x \in \Omega$, this has been pointed out in \cite[Remark 4.2]{RM}. Other examples can be obtained in a very similar way.

\section{Appendix}

Here, we see that the integral of the absolute value of functions is invariant under rearrangement. We have:
\begin{prop} \label{dfun}
Let $\Phi:(a,+\infty) \subset \R^+ \mapsto \R^+$ be a continuous increasing map satisfying 
$$
\lim_{t\to \infty} \Phi(t)=0 \quad \textrm{ and } \quad \lim_{t\to a+} \Phi(t)=\infty.
$$

Then, 
$$
\int_{\Omega^*} \Phi(u^*) dx = \int_\Omega \Phi(|u|) dx = \int_{\Omega^*} \Phi(u_*) dx.
$$
\end{prop}
\begin{proof}
It is a direct consequence of the layer-cake formula given for instance at \cite[Theorem 10.1]{Beng}. We choose $m(dx)=dx$ setting $\Phi(t) = \nu([0,t)^c)$.
\end{proof}

\begin{remark}
An analogous result holds if we assume that $\Phi$ is increasing and satisfies $\Phi(0)=0$. See \cite[Theorem 3.6]{Beng}.
\end{remark}

\vspace{0.5 cm}

{\bf Acknowledgements.} 
RB has been supported by Fondecyt (Chile)
Project  \# 120--1055, MCP by CNPq 308950/2020-8, and FAPESP 2020/14075-6 and 2020/04813-0 (Brazil); MS has been supported by Fondecyt Regular \#1190388. Finally, 
we would like to mention that this work was partially done while MCP was visiting the Instituto de F\'isica at P. Universidad Cat\'olica de Chile. 
He kindly expresses his gratitude to the institute.

\end{document}